\documentclass[11 pt]{amsart}
\usepackage[normalem]{ulem}
\usepackage{soul}
\usepackage[margin = 1 in]{geometry}
\usepackage{
  amsthm,
  amsmath,
  amssymb,
  paralist,
  graphicx,
  hyperref,
  thmtools, 
  xypic, 
  mathrsfs
  }

\usepackage{color}
\definecolor{purple}{cmyk}{0.05,0.8,0,0}

\ifcsname theorem\endcsname{}\else\declaretheorem[parent=section]{theorem}\fi
\ifcsname corollary\endcsname{}\else\declaretheorem[sibling=theorem]{corollary}\fi
\ifcsname lemma\endcsname{}\else\declaretheorem[sibling=theorem]{lemma}\fi
\ifcsname proposition\endcsname{}\else\declaretheorem[sibling=theorem]{proposition}\fi
\ifcsname conjecture\endcsname{}\else\fi
\ifcsname problem\endcsname{}\else\fi
\ifcsname question\endcsname{}\else\declaretheorem[sibling=theorem]{question}\fi
\ifcsname definition\endcsname{}\else\fi
\ifcsname exercise\endcsname{}\else\fi
\ifcsname example\endcsname{}\else\fi
\ifcsname remark\endcsname{}\declaretheorem[sibling=theorem, style=remark]{remark}\fi

\providecommand {\Q}{{\bf Q}}

\renewcommand {\P}{{\bf P}}

\providecommand {\A}{{\bf A}}

\providecommand {\from}{{\colon}}

\providecommand{\Aut}{\operatorname{Aut}}

\providecommand{\Pic}{\operatorname{Pic}}

\newcommand{\Hdg}{\mathcal{H}_{d,g}}
\renewcommand{\H}{\mathcal{H}_{3,g}}
\newcommand{\Hs}{\mathcal{H}^{s}_{3,g}}
\newcommand{\cC}{\mathcal{C}}
\renewcommand{\cR}{\mathcal{R}}
\newcommand{\cP}{\mathcal{P}}

\newcommand{\cT}{\mathcal{T}}
\newcommand{\tT}{\widetilde{\cT}}
\newcommand{\Hf}{\mathcal{H}^{\dagger}_{3,g}}
\newcommand{\Hfs}{\mathcal{H}^{\dagger s}_{3,g}}
\newcommand{\cM}{\mathcal{M}}
\newcommand{\cN}{\mathcal{N}}
\renewcommand{\O}{\mathcal{O}}
\renewcommand{\cH}{\mathcal{H}}
\newcommand{\Rs}{\cR^{s}}
\newcommand{\M}{\cM_{g}}
\newcommand{\cV}{\mathcal{V}}
\newcommand{\F}{{\bf F}}
\newcommand{\G}{{\mathcal G}}

\newcommand{\cI}{\mathcal{I}}
\newcommand{\cJ}{\mathcal{J}}
\newcommand{\Spec}{{\rm Spec}\,}
\renewcommand{\Pic}{{\rm Pic}\,}
\renewcommand{\Q}{{\bf Q}}
\newcommand{\pr}{{\rm pr}}
\newcommand{\Br}{{\rm Br}}

\declaretheorem[title=Theorem]{maintheorem}

\newtheorem{remark}[theorem]{Remark}

\title{On the Chow ring of the Hurwitz space of degree three covers of $\P^1$}
\author{Anand Patel \& Ravi Vakil}
\begin{document}
\maketitle


\begin{abstract}
We show that the Chow ring (with $\Q$-coefficients) of the Hurwitz space parametrizing degree $3$ covers of $\P^{1}$ is tautological, answering a question posed by N. Penev and the second author in \cite{penev-vakil}. We also compute the rational Picard groups of auxiliary spaces of degree $3$ maps with special marked points..   
\end{abstract}

\section{Introduction}

In this paper we determine the Chow ring (with $\Q$-coefficients) of the Hurwitz space $\H$ parametrizing branched covers $\alpha \from C \to \P^1$ of degree $3$ and genus $g(C) = g$ up to automorphisms of $\P^1$.  Throughout, we consider Chow rings and Chow groups with $\Q$-coefficients. 
   
  Our main theorem is:
   \begin{maintheorem}\label{thm:main}
     The Chow ring $A^{*}(\H)$  is generated by
         $\kappa_1$, and hence tautological.
   \end{maintheorem}

   \begin{remark} S. Canning and H. Larson
       have now shown that $A^*(\H) = \Q[\kappa_1]/(\kappa_1^{n_g})$
       with $n_g=3$ for $g \geq 6$, $n_g=2$ for $3 \leq g \leq 5$, and
       $n_2=1$, see \cite[~Theorem 1.1(1)]{Canning-Larson}.  In the process, they correct
       an error in an earlier version of this paper.
\end{remark}

To put \autoref{thm:main} in context, we summarize known results about the algebraic cohomology of the Hurwitz spaces $\Hdg$.  Guided by a now-standard philosophy, we can define a {\sl tautological ring} $R^{*}(\mathcal{H}_{d,g})$ as a subring of the Chow ring $A^{*}(\Hdg)$ generated by certain ``intrinsic'' classes constructed from the universal branched cover diagram 
$$
\xymatrix @C -1.1pc  {
\cC \ar[rr]^{\alpha} \ar[rd]_{\phi}&& 
\cP \ar[ld]^\pi  \\ 
& \Hdg} 
$$
 
 A variant of these classes was   first studied
  \cite{kazaryan-lando}, where they were called {\sl basic}
  classes. Among the tautological classes are the well-known kappa
  classes $\kappa_{i} := \phi_{*}(c_{1}^{i+1}(\omega_{\phi}))$, but a
  priori there are more.

A speculative geometric description of relations for $R^{*}(\Hdg)$
involves the natural stratification of $\Hdg$ by the {\sl Hurwitz
  strata} -- loci parametrizing branched covers with prescribed
branching behavior.  Kazaryan and Lando showed in
\cite{kazaryan-lando} that the fundamental classes of many Hurwitz
strata are basic. In parallel with Faber's relations in the tautological ring of $\M$ (see the discussion after Conjecture 2 of \cite{faber-tautological}), it is reasonable to wonder if such descriptions for {\em empty} strata, which give relations, might give {\em all} relations.    For example, since it is clearly impossible for a degree four map to possess a branch point of order $5$, the fundamental class of the corresponding Hurwitz stratum will be zero.  In the few cases we understand, these are all the relations.  It would be interesting to know whether all relations are obtained in this way.

 The open subset of $\Hdg$ parameterizing simply-branched coverings is called the {\sl small} Hurwitz space $\Hdg^{s}$. In all currently understood cases, the Chow groups of  $\Hdg^{s}$ are trivial.  In codimension one the ``Picard rank conjecture'' \cite{dp:pic_345} predicts that $A^{1}\Hdg^{s} = 0$. This conjecture has been verified when the degree $d \leq 5$ in \cite{dp:pic_345} (where unirational descriptions are used to study $\Hdg$) and also for degrees $d \geq 2g+1$ in \cite{mochizuki95:_hurwit}, where it is shown to be a consequence of Harer's theorem \cite{harer:Pic_Mg} on the Picard group of $\cM_{g}$.  In all cases mentioned, the equality $R^{1}(\Hdg) = A^{1}(\Hdg)$ holds, and both groups are generated by fundamental classes of divisorial Hurwitz strata: the {\sl Maxwell} stratum parametrizing covers having two ramification points mapping to the same point, and the {\sl caustic} stratum parametrizing covers having a point of nonsimple ramification.   (This language is used by Kazaryan and Lando, and is standard in the singularity theory community.  When $d \geq 4$ these two divisorial strata are known to be linearly independent \cite{dp:pic_345}.)

 So we may pose the general question:
 \begin{question}\label{question-trivial}
Is $A^{*}(\Hdg^{s}) = \Q$ for all $d, g$?
 \end{question}

As far as the ring structure of $A^{*}(\Hdg)$ is concerned, essentially nothing is known.  In the case of hyperelliptic curves, one easily sees that $A^{*}(\cH_{2,g}) = \Q$. Indeed, (up to stack-theoretic considerations which are irrelevant when dealing with $\Q$-coefficients) the moduli space $\cH_{2,g}$ is indistinguishable from the space $\cM_{0,2g+2}$, which has trivial Chow ring.  Hence, \autoref{thm:main} is the next open case to consider.

Our current understanding of the global algebraic geometry (cohomology, unirationality, point counting, etc.) of Hurwitz spaces is roughly divided into two regimes.  The {\sl low degree regime} consists of spaces parametrizing covers of degree at most $5$, and is accessible due to classical unirational parametrizations. Many basic questions (for example \autoref{question-trivial}) still remain open in this regime, but the explicit classical parametrizations provide us with some leverage which we do not have for larger $d$.  

The {\sl high degree regime} consists of Hurwitz spaces satisfying the bound $d \geq 2g+1$. In this case, the ``tower'' 
$$
\Hdg \to \Pic^{d} \to \M
$$
exhibits $\Hdg$ as a smooth fibration over $\M$. The fibration allows
us, ideally, to transfer information we may have about $\M$ to
information about $\Hdg$. (See for instance
\cite{mochizuki95:_hurwit}, and also \cite{edidin-severi} for a
different use of the ``tower''.)  Given the stabilization of the
cohomology ring of $\M$ to the polynomial ring $\Q[\kappa_{1},
\kappa_{2}, \kappa_{3}, ...]$ as $g \to \infty$ \cite{harer:Pic_Mg,
  madsen-weiss}, it seems within reach to address an asymptotic
version of \autoref{question-trivial} --- see section \ref{further}.

\begin{remark} S. Canning and H. Larson
have since shown in \cite{Canning-Larson} the striking fact that $A^i(\mathcal{H}_{4,g}^{s}) =
0$ for $1 \leq  i \leq  \frac {g-13} 4$ and $A^i(\mathcal{H}_{5,g}^{s}) =
0$ for $1 \leq  i \leq  \frac {g-76} 5$, see [CL21, Theorem 1.12].
\end{remark}

The {\sl intermediate degree range} $6 \leq d \leq 2g$ poses great challenges. However, there may still be hope in answering \autoref{question-trivial} with topological techniques as in the case of Harer's theorem \cite{harer:Pic_Mg}. Indeed, $\Hdg^{s}$ is a covering space of $\cM_{0,b}$ via the branch map, and thus it is uniformized by the Teichmuller space $\cT_{0,b}$. The braid group $\Br_{b}$ acts freely on $\cT_{0,b}$, and the orbifold quotient is $\cM_{0,b}$.  This means there is a finite index subgroup $\Gamma \subset \Br_{b}$ such that $\cT_{0,b} / \Gamma = \Hdg^{s}.$ Thus, one might concievably translate questions about the cohomology of $\Hdg^{s}$ into statements about the group cohomology of $\Gamma$.

This puts \autoref{thm:main} in context.  In the degree $3$ setting, Hurwitz strata are very easy to describe and obey very simple closure relations: $\cT_{k}$ parametrizes covers having at least $k$ triple ramification points, and $\cT_{k+1}$ is a divisor in $\cT_{k}$.  Furthermore, it is well known that the divisor class $[\cT_{1}]$ is a nonzero multiple of $\kappa_{1}$.

We also prove a version of \autoref{thm:main} for the {\sl framed} Hurwitz space $\Hf$ parametrizing degree $3$ covers of a fixed $\P^1$. There is a natural map $$q \from \Hf \to \H$$ which is a quotient by $\text{Aut}\, \P^{1} = PGL_{2}(k)$. Specifically, we have

\begin{maintheorem}\label{thm:pullback}
  The pullback homomorphism  
      $$q^{*} \from A^{*}(\H) \to A^{*}(\Hf)$$
  is an isomorphism. 
\end{maintheorem}

Finally, we arrive at an answer to \autoref{question-trivial} in the $d = 3$ case: 

\begin{maintheorem}\label{thm:trivial}
We have $A^{*}(\H^{s}) = A^{*}(\H^{\dagger s}) = \Q$.
\end{maintheorem}

The proof of \autoref{thm:main} relies on the following ingredients: 
\begin{enumerate}
\item The classical description of trigonal curves as curves on Hirzebruch surfaces.
\item  The {\sl Maroni stratification} of $\H$, and a result of the second author and N. Panev describing the Chow groups of the open Maroni strata.
\item Vistoli's theorem describing the pullback map $q^{*} \from
  A^{*}(X/G) \to A^{*}(X)$ associated to a quotient $q \from X \to
  X/G$ of a variety $X$ by an algebraic group $G$.  
\end{enumerate}

Briefly, ingredients $(1), (2),$ and $(3)$ help us understand the
Maroni stratification, which allows us to prove that $\kappa_{1}$
generates the Chow ring. 

\subsection{Remark on integral coefficients}
We point out that the remarkable paper of  Bolognesi and Vistoli
\cite{Bolognesi-Vistoli} determines the {\em integral} Picard group of
the space of trigonal curves (in all genera).
There seems no straightforward way to combine their work with ours, to get at the integral Chow groups in higher codimension.  
\begin{remark}  
S. Canning and H. Larson
    have now determined the integral Picard group for the spaces of degree
    $4$ and $5$ covers as well, see \cite[~Theorem 1.8.]{Canning-Larson}.
    \end{remark}

We also mention that the integral Chow ring of the moduli stack of hyperelliptic curves (of even genus) and the stack of cyclic covers of $\P^{1}$ have been computed in \cite{edidin-fulghesu-hyperelliptic} and \cite{fulghesu-viviani-chowcyclic}, respectively. In particular, these papers show that there are nontrivial torsion classes in the integral Chow ring.

\subsection{Outline of the paper} 
We begin in section \ref{section:prelim} by recalling classical facts about trigonal curves, particularly the stratification by Maroni invariant.   This leads to the description of (an open subset of) $\H$ as a quotient of an affine variety $\cV$ by the action of a connected algebraic group $G$. We end section \ref{section:prelim} by recalling Stankova's theorem stating that the group $A^{1}(\H)$ is freely generated by $\kappa_{1}$. 

Section \ref{section:vistoli-quotient}  begins by recalling a theorem of A. Vistoli \cite{vistoli-quotient} describing Chow rings of quotients. We apply this theorem to the $PGL_{2}(k)$ quotient map $q \from \Hf \to \H$ to deduce that the kernel of the pullback map $q^{*} \from A^{*}(\H) \to A^{*}(\Hf)$ is the ideal generated by the second Chern class of a vector bundle $V$ on $\H$, whose projectivization is $\cP$, the universal target $\P^1$. 

Section \ref{section:univ-ram-point-R}  is devoted to showing that the class $c_{2}(V)$ is supported on the divisor $\cT_{1} \subset \H$. We accomplish this by proving that $A^{1}(\Rs) = 0$, where $\Rs$ is the ``universal simple ramification point'' which is a degree $2g+4$ etale cover of $\Hs$.  

Section \ref{section:univ-trip-ram-point}  is devoted to proving that $A^{1}(\tT_{1})$ is generated by $\kappa_{1}$. Here $\tT_{1}$ denotes the moduli space of branched triple covers with a marked point of triple ramification -- it is the normalization of $\cT_{1}$.  The techniques used are similar to those of section \ref{section:univ-ram-point-R}.  When combined with the results from sections \ref{section:vistoli-quotient}  and \ref{section:univ-ram-point-R}, we deduce that $c_{2}(V)$ is a multiple of the class $\kappa_{1}^{2}$.

In section \ref{section:framed-chow-ring}, we show that the fundamental classes of the Maroni strata in $\Hf$ are expressible in kappa classes.  This, along with ingredient number $2$ mentioned above, allows us to conclude that $A^{*}(\H)$ is generated by the kappa classes.

\subsection{Notation and conventions} We work over an algebraically closed field $k$ of characteristic $0$. All Chow rings and tautological rings will be taken with coefficients in $\Q$. As a result, the distinction between ``stack'' and  ``coarse space'' will be irrelevant for us.

The following spaces will appear prominently in the paper: 

\begin{enumerate}
\item [$\H$:] This is the Hurwitz space parametrizing genus $g$, degree $3$ covers of $\alpha \from C \to \P^1$, with $C$ smooth, up to automorphisms of $\P^1$.
\item [$\Hs$:] This is the open substack of $\H$ parametrizing those covers which are simply branched.
\item [$\Hf$:] This is the stack parametrizing genus $g$, degree $3$ covers of a {\sl fixed} $\P^1$, with smooth source curve.
\item [$\Hfs$:] This is the open substack of $\Hf$ parametrizing simply branched covers.
\item [$\M$:] This is the moduli space of smooth genus $g$ curves.
\end{enumerate}
In general, the superscript $s$ will denote the adjective ``simple'', and a $\dagger$ will denote the adjective ``framed''.  Also, the superscript $\circ$ will stand for ``minimal Maroni invariant''.

\subsection*{Acknowledgments} We thank Gabriel Bujokas,
  Francesco Cavazzani, Anand
Deopurkar, and Joe Harris  for valuable conversations.
We thank in particular Samir Canning and Hannah Larson for important
conversations, for catching a serious error in an earlier version
of this paper, and for significantly improving on these results in \cite{Canning-Larson}.
We thank the anonymous referee for several helpful suggestions and corrections.

\section{Preliminaries}\label{section:prelim}
\subsection{Trigonal curves as divisors on Hirzebruch surfaces} \label{Maroni}
It is a classical fact that every degree three cover $\alpha \from C
\to \P^1$ has a canonical realization as a curve $C \subset \F_{m}$ in
a Hirzebruch surface $\F_{m}$, with the projection $\pi \from \F_{m}
\to \P^1$ inducing the map $\alpha$.  Geometrically, the Hirzebruch
surface $\F_{m}$ can be realized as the surface swept out by the lines
$l_{D} \subset \P(C, \mathcal{K}_C)^\vee \cong \P^{g-1}$ spanned by the divisors $D$ of the given $g^{1}_{3}$.  

\subsubsection{The Maroni invariant} The {\sl Maroni invariant} of a degree three cover $\alpha \from C \to \P^1$ is the integer $m$ where $C \subset \F_{m}$ \cite{maroni}.  Define the closed substack of $\cH_{3,g}$:
$$\cN_{k} := \{ \alpha \in \H \mid \text{the Maroni invariant of $\alpha$ is $\geq k$}\}.$$

The behavior of the Maroni invariant is summarized in the following well-known proposition:

\begin{proposition} \label{mar:properties}
\begin{enumerate}
\item The Maroni invariant $m$ does not exceed $(g+2)/3$. 
\item The Maroni invariant is upper semi-continuous on $\H$. When $g$ is even, $m$ takes on any even value from $0$ to $\lfloor(g+2)/3 \rfloor$. When $g$ is odd, $m$ takes on any odd value from $1$ to $\lfloor(g+2)/3 \rfloor$. 
\item When nonempty, the codimension of $\cN_{k}$ is ${\rm max}\,\{0, k-1\}$.
\end{enumerate}
\end{proposition}

\subsection{Expressing an open subset of $\H$ as a quotient}\label{section:quotient-Hurwitz} 
Let $\H^{\circ} \subset \H$ be the open subset parametrizing covers with minimal Maroni invariant ($\cH^\circ_{3,g} = \cH_{3,g} \setminus \cN_2$). We can represent $\H^{\circ}$ as a quotient in a straightforward way. 

Let $\F$ denote either $\F_{0}$ or $\F_{1}$, depending on whether $g$ is even or odd, respectively.  Let $\pi \from \F \to \P^1$ be a selected projection.   Then there is a unique curve class $C$ on $\F$ parametrizing genus $g$ curves which map with degree $3$ under the projection $\pi$. (This is clearly true for any Hirzeburch surface, not just the two considered in this section.)

Consider the linear system $|C| \simeq \P^{N}$ with $\Delta \subset \P^{N}$ the discriminant hypersurface parametrizing singular curves.  The complement $$U_{g} := \P^{N} \setminus \Delta$$ parametrizes all smooth trigonal curves with minimal Maroni invariant. 

There is a natural surjective map 
$$
u \from U_{g} \to \H^{\circ}
$$

The map $u$ is a quotient map: $\Aut \F$ acts on the variety $U_{g}$ with finite stabilizers, and $\H^{\circ}$ is the geometric quotient.  The complement $\H \setminus \H^{\circ}$ is a divisor only when $g$ is even -- in the odd case, this complement has codimension two. 

\subsection{The divisor theory of $\H$} In \cite[Thm.~IV]{stankova}, Stankova proves: 
\begin{theorem}[Stankova]\label{divH}
The group $A^{1}(\H)$ is freely generated by $\kappa_{1}$.
\end{theorem}

The theorem also holds for $\Hf$.  These statements are all consequences of a general theorem on Chow rings of quotient varieties due to A. Vistoli.  We review this theorem in the next section.

\section{Vistoli's theorem: Chow rings of quotients}\label{section:vistoli-quotient}

\subsection{Vistoli's theorem}

 We recall a theorem of A. Vistoli: 
\begin{theorem}[Vistoli] \label{vistoli}
Let $q \from X \to Y$ be a (categorical) quotient of a smooth stack $X$ by $G = GL_{n}(k)$ or $SL_{n}(k)$. and let $V := X \times \A^{n} / G$ be the canonically defined vector bundle on $Y$. (Note that if  $G = SL_{n}$, then $c_{1}(V) = 0$.) Then: 
\begin{enumerate}
\item[(i)] The pullback map $q^{*} \from A^{*}(Y) \to A^{*}(X)$ is surjective.
\item[(ii)] The kernel of $q^{*}$ is the ideal generated by all Chern classes $c_{i}(V)$.
\end{enumerate}
\end{theorem}

Part (i) is \cite[Thm.~2]{vistoli-quotient}, and (ii) follows
from that same Theorem,
 by induction on codimension, as in \cite[Thm.~2.4]{penev-vakil}.

In fact, we need the analogous theorem for actions of the group $PGL_{n}(k)$.  The $SL_{n}(k)$ case is essentially equivalent to the $PGL_{n}(k)$ case when we consider $\Q$-coefficients. Indeed, suppose $PGL_{2}(k)$ acts on $X$ (with finite stabilizers).  The stack quotient $[X/SL_{n}(k)]$ is then an etale $\mu_{n}$-gerbe over the stack quotient $[X/PGL_{n}(k)]$, and therefore the Chow rings with $\Q$-coefficients of both quotients coincide. 

Furthermore, the vector bundle $V$ technically only exists over $[X/SL_{n}(k)]$, but we abuse notation and make arguments about it as if it were a bundle over $[X/PGL_{n}(k)]$.  Since we are only concerned with the rational Chern classes of $V$, this abuse will be harmless.

\subsection{The quotient map $q \from \Hf \to \H$ and the class $c_{2}(V)$}\label{section:quotientmap-c2}

Let $q \from \H^{\dagger} \to \H$ be the natural $PGL_2(k)$ quotient map. By \autoref{vistoli}, we see

$$
A^{*}(\H^{\dagger}) = A^{*}(\H) / (c_2(V)).
$$
Here $c_{2}(V)$ is the second Chern class of a vector bundle $V$ on $\H$ with trivial first Chern class, parametrizing the universal $g^{1}_{3}$.  We may think of $c_{2}(V)$ as a class in $A^{2}(\H)$ as follows: Let $W$ be any rank $2$ bundle whose  projectivization is the target $\cP$ in the universal diagram: 
$$
\xymatrix @C -1.1pc  {
\cC \ar[rr]^{\alpha} \ar[rd]_f&&
\cP \ar[ld]^\pi \\
& \H}
$$
 Then we take $c_{2}(V)$ to be the expression $c_{2}(W) - \frac{1}{4}c_{1}^{2}(W)$.  Note that this class is preserved after tensoring $W$ by a line bundle, and therefore does not depend on the choice of $W$.

By Vistoli's theorem \ref{vistoli}, in order to get a handle on the Chow ring of $\H$, we must first understand the class $c_{2}(V)$ and then compute $A^{*}(\Hf)$. 

\autoref{c2} states that $c_{2}(V)$ is tautological.  Its proof relies on an understanding of the divisor theory of the ``universal ramification point'' $\cR$, and the ``universal triple ramification point'' $\tT_{1}$.

\section{The universal ramification point $\cR$}\label{section:univ-ram-point-R}

In this section we study the divisor theory of the ``universal ramification point'' $\cR \subset \cC$ defined functorially as: 
$$\cR = \left\{(\alpha, p) \mid [\alpha \from C \to \P^1] \in \H, \, \text{and} \, p \in C \, \, \text{a ramification point of $\alpha$}\right\}.$$  If the reader prefers, $\cR$ is also the ramification locus of the universal branched covering $\alpha \from \cC \to \cP$.  We let $\phi \from \cR \to \H$ denote the natural, degree $2g+4$ finite projection map. 

The space $\cR$ has a Zariski open substack $\Rs \subset \cR$ defined functorially as: 
$$\Rs = \left\{(\alpha, p) \mid [\alpha \from C \to \P^1] \in \Hs, \, \text{and} \, p \in C \, \, \text{a ramification point of $\alpha$}\right\}.$$ In other words, $\Rs = \phi^{-1}(\Hs).$ 

Our strategy, as usual, will be to express $\cR$ (or an open substack of it) as a quotient, and then use Vistoli's theorem \ref{vistoli}.

\subsection{Expressing an open substack of $\cR$ as a quotient}\label{section:quotient-R} 
Let us fix a Hirzebruch surface $\F = \F_{0}$ or $\F_{1}$, and the appropriate linear system $|C| \simeq \P^{N}$ parametrizing  trigonal curves of genus $g$.  Define the following open subset of the universal ramification point $\cR$: $$\cR^{\circ} := \left\{(\alpha, p) \mid \text{$\alpha$ has minimal Maroni invariant, $p$  a ramification of $\alpha$}\right \}.$$ 
The complement $\cR \setminus \cR^{\circ}$ is a divisor only when $g$ is even, in which case it is the pullback of the Maroni divisor in $\H$, and therefore a multiple of $\kappa_{1}$.  Note that $\cR^{\circ}$ contains the open subset $\Rs \cap \cR^{\circ}$, which we denote by $(\Rs)^{\circ}$.

We now express $\cR^{\circ}$ as a quotient.  Let $W$ be the vector bundle on $\F$ whose fiber at a point $p \in \F$ is given by $$W|_{p} = H^{0}(\F, \O_{\F}(C) \otimes \cI_{2p}).$$ Here, by ``$np$'' we mean the scheme $\Spec k[\epsilon]/(\epsilon^{n}) \subset \F$ with closed point $p$, and lying in the ruling line $L_{p} \subset \F$ determined by $p \in \F$. 

The space $\cR^{\circ}$ is a quotient of a suitable open subset of the projective bundle $\P W$ by the obvious action of the group $G := \Aut \F$. In order to understand the bundle $\P W$ better, it will be necessary first to formally construct $W$. For this, we need to introduce some jet bundles.
\subsubsection{Jet bundles on the Hirzebruch surface $\F$} 

The next subsections are more technical, so we briefly sketch our intentions before proceeding.  Our primary objective is to compute divisor classes of several naturally occurring divisors in the projective bundle $\P W$, which itself needs to be constructed. If we informally think of $\P W$ as parametrizing pairs $(C,p)$ where $\pi: C\to \P^{1}$ is ramified at $p$, then we seek to understand the divisors in $\P W$ parametrizing curves which are either singular at $p$, or have a higher order ramification point at $p$.  In order to compute the divisor classes of these loci, we make extensive use of jet bundles to construct the vector bundle $W$ as well as two vector subbundles of $W$ whose projectivizations provide the two divisors whose classes we need to compute.

To begin, consider the fiber product $\F \times_{\P^{1}} \F$ with its two projections $\pr_{1}$ and $\pr_{2}$.  Let $\Delta \subset \F \times_{\P^1} \F$ be the diagonal, and let $\cI_{\Delta}$ be its ideal sheaf. 

Then the $n$'th order jet bundle of $\O_{\F}(C)$ relative to $\pi \from \F \to \P^{1}$ is defined to be 
\begin{equation}\label{njet}
J_{\pi}^{n}(\O_{\F}(C)) := \pr_{1*}\left(\pr_{2}^{*}(\O_{\F}(C)) \otimes \O_{\F \times_{\P^1} \F}/\cI^{n}_{\Delta}\right).
\end{equation} 

It is easy to see that $J_{\pi}^{n}(\O_{\F}(C))$ is a rank $n$ vector bundle on $\F$, and that there is a natural evaluation map $$ev \from H^{0}(\F, \O_{\F}(C)) \otimes \O_{\F} \to J_{\pi}^{n}(\O_{\F}(C)). $$ Furthermore, the jet bundles $J_{\pi}^{n}(\O_{\F}(C))$ are related by exact sequences 
\begin{equation}\label{filter}
0 \to \omega_{\pi}^{\otimes n-1} \otimes \O_{\F}(C) \to J_{\pi}^{n}(\O_{\F}(C)) \to J_{\pi}^{n-1}(\O_{\F}(C)) \to 0.
\end{equation}

The vector bundle $W$ is defined to be the kernel in the sequence 
$$
0 \to W \to  H^{0}(\F, \O_{\F}(C)) \otimes \O_{\F} \to J_{\pi}^{2}(\O_{\F}(C)) \to 0.
$$
Surjectivity of $ev \from H^{0}(\F, \O_{\F}(C)) \otimes \O_{\F} \to J_{\pi}^{2}(\O_{\F}(C))$ needs to be checked for our particular divisor classes $C$, but this is straightforward and we omit it.  It is also straightforward to check that $ev \from H^{0}(\F, \O_{\F}(C)) \otimes \O_{\F} \to J_{\pi}^{3}(\O_{\F}(C))$ is surjective, and we call its kernel $W_{\text{tr}}$. (Here $tr$ stands for ``triple''.) From sequence \eqref{filter}, we see that there is an exact sequence of vector bundles
 \begin{equation}\label{Wtr}
0 \to W_{\text{tr}} \to W \to \O_{\F}(C) \otimes \omega_{\pi}^{\otimes 2} \to 0.
\end{equation}
 
\begin{remark}
The projective subbundle $\P W_{\text{tr}} \subset \P W$ parametrizes the pairs $(C,p)$ where $p$ is a triple ramification point of $C$.
\end{remark}

Finally, we construct a third jet bundle which we call $J_{\text{fat}}^{3}(\O_{\F}(C))$.  In slight contrast  with the previous construction, we consider the diagonal $\Delta$ as a subscheme of the {\sl absolute} product $\F \times \F$.  We continue referring to the projections by $\pr_{1}$ and $\pr_{2}$. Letting $\cJ_{\Delta}$ denote the ideal sheaf, we define
$$
 J_{\text{fat}}^{3}(\O_{\F}(C)) := \pr_{1*}\left(\pr_{2}^{*}(\O_{\F}(C)) \otimes \O_{\F \times \F}/\cJ^{2}_{\Delta}\right)
$$
(Note: The superscript $3$ indicates that $J_{\text{fat}}^{3}(\O_{\F}(C))$ is a rank $3$ vector bundle on $\F$.)

As before, there is an evaluation map $$ev \from H^{0}(\F, \O_{\F}(C)) \otimes \O_{\F} \to  J_{\text{fat}}^{3}(\O_{\F}(C))$$ which is surjective because the linear system $|C|$ is sufficiently positive in all cases we consider, and therefore separates tangent vectors. The kernel of $ev$ is a vector bundle which we denote by $W_{n}$. ($n$ stands for ``node''.)

\begin{lemma}\label{fatvsrel}
There is an exact sequence of vector bundles on $\F$:
\begin{equation}\label{eq:fatvsrel}
0 \to \pi^{*}(\omega_{\P^1}) \otimes \O_{\F}(C) \to J_{\text{fat}}^{3}(\O_{\F}(C)) \to J_{\pi}^{2}(\O_{\F}(C)) \to 0.
\end{equation}
\end{lemma}

\begin{proof} Consider the closed inclusions $\Delta \subset \F\times_{\P^{1}}\F \subset \F \times \F$. Let $\cI_{\F\times_{\P^{1}}\F}$ be the ideal sheaf of the latter inclusion. Then $\cI_{\F\times_{\P^{1}}\F}$ is invertible, as $\F\times_{\P^{1}}\F$ is a divisor in the smooth variety $\F \times \F$. Keeping with previous notation, let  $\cJ_{\Delta}$ be the ideal sheaf of $\Delta \subset \F \times \F$.  Then consider the sequence 
\begin{equation}\label{seq}
0 \to (\cI_{\F\times_{\P^{1}}\F} + \cJ^{2}_{\Delta})/\cJ^{2}_{\Delta} \to \O_{\F \times \F}/\cJ^{2}_{\Delta} \to \O_{\F \times \F}/(\cI_{\F\times_{\P^{1}}\F} + \cJ^{2}_{\Delta}) \to 0.
\end{equation}

Each sheaf is supported on the diagonal $\Delta$.  We now prove that $(\cI_{\F\times_{\P^{1}}\F} + \cJ^{2}_{\Delta})/\cJ^{2}_{\Delta}$ is isomorphic to the invertible $\O_{\Delta}$-module $\cI_{\F\times_{\P^{1}}\F}/(\cI_{\F\times_{\P^{1}}\F}\cdot \cJ_{\Delta})$.  Note that there is a natural map $f \from \cI_{\F\times_{\P^{1}}\F}/(\cI_{\F\times_{\P^{1}}\F}\cdot \cJ_{\Delta}) \to (\cI_{\F\times_{\P^{1}}\F} + \cJ^{2}_{\Delta})/\cJ^{2}_{\Delta}$. We claim it is an isomorphism. 

For this, it suffices to check this isomorphism affine-locally on $\F$.  Let $\A^{2} \subset \F$ be an open subset with coordinate ring $k[x,y]$, such that the inclusion $k[x] \hookrightarrow k[x,y]$ induces the projection $\pi \from \F \to \P^1$.  Let $p \in \F$ be the origin $(0,0)$. Then, restricted to $p$, the sequence \eqref{seq} becomes the following sequence of $k[x,y]$-modules: 
$$
0 \to (x,y^{2})/(x^{2},xy,y^{2}) \to k[x,y]/(x^{2},xy,y^{2}) \to k[x,y]/(x,y^{2}) \to 0.
$$
The kernel is generated by the residue class $\bar{x} \in (x,y^{2})/(x^{2},xy,y^{2})$.  The residue class $\bar{x}$ is the image of a local generator of the invertible $\O_{\Delta}$-module $\cI_{\F\times_{\P^{1}}\F}/(\cI_{\F\times_{\P^{1}}\F}\cdot \cJ_{\Delta})$ under the map $f$.

It is straightforward to check that the invertible $\O_{\Delta}$-module $\cI_{\F\times_{\P^{1}}\F}/(\cI_{\F\times_{\P^{1}}\F}\cdot \cJ_{\Delta})$ is isomorphic to the line bundle $\pi^{*}(\omega_{\P^{1}})$ when we identify $\F$ with $\Delta$ via projection.  (In terms of the local equations above, the element $\bar{x}$ is a generator of the vector space $\mathfrak{m} / \mathfrak{m}^{2}$ where $\mathfrak{m} \subset \O_{\P^{1},0}$ is the maximal ideal of the point $0$.) With these identifications, and upon tensoring \eqref{seq} by $\O_{\F}(C)$, we obtain \eqref{eq:fatvsrel}.

\end{proof}

Equation \eqref{eq:fatvsrel} implies the existence of a sequence of vector bundles 
\begin{equation}\label{Wn}
0 \to W_{n} \to W \to \pi^{*}(\omega_{\P^1}) \otimes \O_{\F}(C) \to 0.
\end{equation}

\begin{remark}
The bundle $\P W_{n}$ parametrizes pairs $(C, p)$ where $p$ is a singular point of $C$.
\end{remark}

\subsection{The divisor theory of $\cR$}\label{divisorR}

The variety $\P W$ has a projection $$p_{1} \from \P W \to \P^{N} = \P(H^{0}(\F, \O_{\F}(C))),$$ and a second projection $p_{2} \from \P W \to \F$ which exhibits it as a projective bundle over the surface $\F$. Since the rank of the Picard group of $\F$ is two, we get 
$$
A^{1}(\P W) = \Q^{3}.
$$
Let $\P W^{s}_{\text{sm}} \subset \P W$ denote the open subset parametrizing smooth curves $C$ which are simply branched under $\pi \from \F \to \P^1$.  The closed set $\P W \setminus \P W^{s}_{\text{sm}}$ consists of four irreducible divisors, whose generic points have been indicated: 
\begin{enumerate}
\item[$\delta_{0}$:] parametrizes pairs $(C,p)$ where $C$ has a node which is away from $p$.
\item[$\delta_{n}$:] parametrizes pairs $(C,p)$ where $p$ is a node of $C$. It is the subbundle $\P W_{n} \subset \P W$, from \eqref{Wn}.
\item[$\xi_{\text{tr}}$:]  parametrizes pairs $(C,p)$ where $p$ is a point of triple ramification. It is the subbundle $\P W_{\text{tr}} \subset \P W$ from \eqref{Wtr}.
\item[$\tau$:] parametrizes pairs $(C,p)$ where $C$ has a point of triple ramification away from $p$.
\end{enumerate}

If we let $\delta \subset \P^{N}$ denote the discriminant divisor parametrizing singular curves in the linear system $|C|$, one can check that $$p_{1}^{*}\delta = \delta_{0} + 3\delta_{n}.$$ (The coefficient $3$ comes from an analytic computation. This computation is equivalent to the one needed to show that as a plane curve specializes to a nodal curve, each branch of the node is the limit of $3$ flex points.  We will not use the specific coefficient of $\delta_{n}$ in our arguments.)  For simplicity, we suppress the $p_{1}^{*}$ and will refer to $\delta \subset \P W$.  We now prove: 
\begin{proposition}\label{indepW}
The divisor classes $[\delta], [\delta_{n}],$ and $[\xi_{\text{tr}}]$ are independent in $A^{1}(\P W)$. Therefore, $$A^{1}(\P W^{s}_{\text{sm}}) = 0.$$
\end{proposition}

\begin{proof}
The natural projection $p_{1} \from \P W \to |C| = \P^{N}$ induces a hyperplane class $$h \in A^{1}(\P W).$$ Furthermore, since there is a natural inclusion $W \subset k^{N+1} \otimes \O_{\F}$ inducing $p_{1}$, we see that $h$ represents the natural $\O_{\P W}(1)$ on the projective bundle $\P W$.  (In other words, the sections of the line bundle $\O_{\P W}(h)$ are given by $H^{0}(\F, W^{*})$.)  The total boundary $\delta$ is pulled back from the projection $p_{1}$, and therefore its divisor class is a multiple of $h$.  

The divisors $\delta_{n}$ and $\xi_{\text{tr}}$ are the subbundles $\P W_{n}$ and $\P W_{\text{tr}}$, respectively.   By considering the exact sequences \eqref{Wtr} and \eqref{Wn} which give rise to these subbundles, we get 
\begin{eqnarray*}
\delta_{\text{tr}} &=& h + p_{2}^{*}(C + 2\omega_{\pi})\\
\xi_{\text{tr}} &=& h + p_{2}^{*}(C + \pi^{*}\omega_{\P^1}).
\end{eqnarray*}
Therefore, in order for $\delta$, $\delta_{n}$, and $\xi_{\text{tr}}$ to be dependent, we would need the divisor class $C + 2\omega_{\pi}$ to be a multiple of $C + \pi^{*}\omega_{\P^1}$. This is not the case for our trigonal genus $g$ divisor classes $C$.  
\end{proof}

Divisors $\xi_{\text{tr}}$ and $\tau$ descend to give divisors on $\cR$ (by taking closures), and we will use the same notation to refer to these. (Note that $\xi_{\text{tr}}$ is the branch divisor of the finite, degree $2g+4$ forgetful map $\phi \from \cR \to \H$.)  In what follows, we let $\Rs_{1} \subset \cR$ denote $\phi^{-1}(\H^{s})$. 

 \begin{corollary}\label{divR}
 The Chow group $A^{1}(\Rs_{1})$ is trivial. The Chow group $A^{1}(\cR)$ is generated by $\kappa_{1}$ and $\xi_{\text{tr}}$.
 \end{corollary}
\begin{proof}

We consider the fiber square
\begin{equation} \label{sq0}
\xymatrix @C -.6pc  {
\P W^{s}_{\text{sm}} \ar[rr]^{p_{1}} \ar[d]_{q_{\Rs_{1}}}&&
 (\P^{N})^{s}_{\text{sm}}  \ar[d]_{q}  \\
 (\Rs_{1})^{\circ} \ar[rr]^{\phi} && (\Hs)}
 \end{equation} 
 (Here $(\P^{N})^{s}_{\text{sm}}$ parameterizes the smooth and simply branched curves in the linear system $|C|$.) The map $p_{1}$ is obviously $G$-equivariant where $G = \Aut \F$, which implies the following equality in $A^{1}(\cR_{1}^{s \circ})$: $$\ker q_{\Rs_{1}}^{*} = \phi^{*}(\ker q^{*}) = 0.$$ (The second equality follows from \autoref{divH}.) Therefore, by Vistoli's theorem we conclude that $$ A^{1}((\Rs_{1})^{\circ}) = 0.$$  
 
 Now we recall that when $\cR \setminus \cR^{\circ}$ is a divisor (the Maroni divisor), it is a positive multiple of $\kappa_{1}$.   
 We conclude by noting that $\cR \setminus \Rs$ is also a positive multiple of $\kappa_{1}$ and contains two irreducible components: $\xi_{\text{tr}}$, and $\tau$.

\end{proof}

\begin{corollary}\label{supp}
Given any two line bundles $L_{1}, L_{2}$ on $\cR$, the class $\phi_{*}(c_{1}(L_{1}) \cdot c_{1}(L_{2})) \in A^{2}(\H)$ can be represented by a divisor class in  $\cT_{1} \subset \H$.
\end{corollary}

\begin{proof}
This follows from \autoref{divR}. Indeed, all divisor classes on $\cR$ are supported on the locus $\phi^{-1}(\cT_{1}) \subset \cR$.
\end{proof}

\begin{corollary}\label{suppV}
The Chern class $c_{2}(V) \in A^{2}(\H)$ can be represented by a divisor class in $\cT_{1}$.
\end{corollary}

\begin{proof}
Consider the pullback $\phi^{*}(V)$ to $\cR$.  Recall that $\cP = \P V$. The projective bundle $\phi^{*}(\cP)$ over $\cR$ has a section $\sigma$. This section is induced from the graph of the natural map $$\alpha \from \cR \to \cP.$$ 

The section $\sigma$, in turn, induces an exact sequence $$0 \to L_{1} \to \phi^{*}(V) \to L_{2} \to 0$$ for some line bundles $L_{1}$ and $L_2$ on $\cR$.  The corollary now follows from the push pull formula, and from \autoref{supp}.
\end{proof}

\section{The universal triple ramification point $\tT_{1}$}\label{section:univ-trip-ram-point}

Let $\cT_{1} \subset \H$ be the  closed, irreducible, codimension $1$ substack parametrizing covers $\alpha \from C \to \P^1$ with non-simple ramification. $\cT_{1}$ is singular precisely along the locus $\cT_{2} \subset \cT_{1}$ parametrizing covers with at least two non-simple ramification points.

We can normalize $\cT_{1}$ by introducing the space $\tT_{1}$ of triple covers with a {\sl marked} point of non-simple ramification: 

$$\tT_{1} := \{(\alpha \from C \to \P^1, p \in C) \mid \text{$p$ is a point of non-simple ramification}\}.$$

There is a natural forgetful map $$\varphi \from \tT_{1} \to \H$$ which serves as the normalization of $\cT_{1}$ at the level of coarse spaces.  Furthermore, given that we are using ${\bf Q}$ coefficients, it is clear that $$\varphi_{*}(A^{1}(\tT_{1}) )= A^{1}(\cT_{1}).$$

The main result of this section is: 
\begin{theorem}\label{divT}
$A^{1}(\tT_{1})$ is generated by $\kappa_{1}$.
\end{theorem}

 The proof of \autoref{divT} parallels the proof of \autoref{divR}. We study the open set $$\tT_{1}^{\circ} \subset \tT_{1}$$ consisting of covers with minimal Maroni invariant by exhibiting it as a quotient of a space which is easily understood.
 
 The complement, $\tT_{1} \setminus \tT_{1}^{\circ}$ is a divisor only when $g$ is even, in which case it is the pullback of the Maroni divisor in $\H$, and is therefore already known to be a multiple of $\kappa_{1}$. So it suffices to show that the divisor classes on $\tT_{1}^{\circ}$ are multiples of $\kappa_{1}$.



\subsection{Expressing $\tT_{1}^{\circ}$ as a quotient}\label{section:quotient-triple-ram}

 Fix a balanced Hirzebruch surface $\F,$ (either $\F_{0}$ or $\F_{1}$) and let $L$ denote the ruling line class of the  projection $\pi \from \F \to \P^1$. Consider the vector bundle $U$ on $\F$ whose fiber at a point $p \in \F$ is $$U|_{p} = H^{0}(\F, \O_{\F}(C) \otimes \cI_{3p}).$$ As before, we let ``$np$'' denotes the scheme $\Spec k[\epsilon]/(\epsilon^{n})$ with closed point $p$ and lying in the ruling line through $p \in \F$.  $C$ will denote a smooth genus $g$ trigonal curve in $\F$.  

The projectivization $\P U$ maps to the linear system $|\O_{\F}(C)| \simeq \P^{N} $, and is birational onto the locus $Z \subset \P^{N}$ defined generically as 
$$Z := \{ C \mid \text{$C$ possesses a triple ramification point}\}.$$
 Let $p_{1} \from \P U \to \P^{N}$ denote the natural projection. The group $G: = \Aut \, \F$ acts on the linear system $|\O_{\F}(C)|$ and leaves the locus $Z$ invariant.

Let $ \P U_{\text{sm}} \subset \P U$ denote the open set parametrizing smooth curves.  We let $Z_{\text{sm}} \subset Z$ and $\P^{N}_{\text{sm}} \subset \P^{N}$ denote the respective open subsets parametrizing smooth curves. Since $\P U_{\text{sm}}$ is smooth, and since $p_{1}: \P U_{\text{sm}} \to Z_{\text{sm}}$ is finite and birational, it follows that $\P U_{\text{sm}}$ is the normalization of $Z_{\text{sm}}$, and therefore the $G$-action on $Z_{\text{sm}}$ naturally lifts to a $G$-action on  $\P U_{\text{sm}}$. 

In complete analogy with \eqref{sq0}, we have a Cartesian square of spaces:
$$
\xymatrix @C -.6pc  {
\P U_{\text{sm}} \ar[rr]^{p_{1}} \ar[d]_{q_{\tT_{1}}}&&
 \P^{N}_{\text{sm}}  \ar[d]_{q}  \\
 \tT_{1}^{\circ} \ar[rr]^{\phi} && (\H)^{\circ}}
$$
 which expresses $\tT_{1}^{\circ}$ as quotient of $\P U_{\text{sm}}$ by the group $G = \Aut \F$. Clearly $p_{1}$ is $G$-equivariant.

\subsection{The discriminant $\delta \subset \P U$}

$\P U$ has its natural projection $p_{2}$ to $\F$, expressing it as a projective bundle.  Therefore, $$A^{1}(\P U) = \Q^{3}.$$  Of course, $\P U$ can also be described as 
$$\P U := \left\{(C,p) \in \P^{N} \times \F \mid \text{$p$ is a point of triple ramification for the projection $\pi \from C \to \P^1$}\right\}.$$

The discriminant hypersurface $\delta \subset \P U$, pulled back from $\P^{N}$ via $p_{1}$, parametrizes singular curves, and breaks up into three irreducible components: 
 
\begin{enumerate}
\item[$\delta_{\text{red}}$:] This is the locus of {\sl reducible curves}, where a ruling line of $\F$ ``splits off.''  Generically, these are nodal curves $C' \cup L$ where $L$ is a ruling line of $\F$, and $C'$ is a smooth residual trigonal curve, with the marked point $p$ lying on a general point of $L$. This divisor is contracted under the projection $p_{1} \from \P U\to \P^{N}$. (The position of the point $p$ on $L$ is forgotten.)
\medskip
\item[$\delta_{\text{ram}}$:] This generically parametrizes a nodal curve with one branch of the node being tangent to a ruling line. We say such a curve possesses a {\sl ramified node}. The marked point $p$ is the ramified node.
\medskip
\item[$\delta_{0}$:] This generically parametrizes nodal curves whose nodes do not possess branches that are tangent to any ruling line. The marked point $p$ is a smooth triple ramification point of $C$.
\end{enumerate}

\begin{proposition}\label{indep}
The divisor classes $\delta_{0}$, $\delta_{\text{red}}$, and $\delta_{\text{ram}}$ are $\Q$-linearly independent in $A^{1}( \P U)$. 
\end{proposition}

We postpone the proof of \autoref{indep} until \autoref{deltaram}. For now, we make the following observation: Since the total discriminant $\delta$ is a linear combination of $\delta_{0}$, $\delta_{\text{ram}}$, and $\delta_{\text{red}}$ with all coefficients nonzero, it suffices to show that $\delta, \delta_{\text{red}},$ and $\delta_{\text{ram}}$ are linearly independent. 

As in \autoref{divisorR}, we let $h$ denote the hyperplane class on $\P U$, induced by the projection $p_{1} \from \P U \to \P^N$. It is easy to see that $\delta_{\text{red}}$ and $\delta_{\text{ram}}$ are linear subbundles of $\P U$.  Therefore, both divisor classes can be written in the form $h + p_{2}^{*}(D)$ for some divisor class $D \in \Pic \F$.  Furthermore, the total discriminant $\delta$ is evidently pulled back from $\P^{N}$, so $[\delta] = m\cdot h$ for some positive integer $m$.  In order to prove \autoref{indep}, we will explicitly calculate the classes $[\delta_{\text{red}}]$ and $[\delta_{\text{ram}}]$, just as we did in \autoref{divisorR}.  In order to do so, we must consider jet bundles once more.

\subsubsection{Jet bundle construction of $\delta_{\text{red}}$}

Recall the jet bundle $J_{\pi}^{n}(\O_{\F}(C))$ from \eqref{njet}. In particular, consider the cases $n =3$ and $4$, and consider the following sequence relating them: 
\begin{equation}\label{reljetseq}
0 \to \O_{\F}(C)\otimes \omega^{3}_{\pi} \to J_{\pi}^{4}(\O_{\F}(C)) \to J_{\pi}^{3}(\O_{\F}(C)) \to 0
\end{equation}
Recall that the vector bundle $U$ is formally defined as the kernel in the following sequence: 
$$
0 \to U \to H^{0}(\F, \O_{\F}(C))\otimes \O_{\F} \to J_{\pi}^{3}(\O_{\F}(C)) \to 0.
$$
where the right hand map is ``evaluation''.
Given the existence of sequence \eqref{reljetseq}, we find that $\delta_{\text{red}}$ is the projective subbundle $\P U_{\text{red}}$, where $U_{\text{red}}$ occurs in the sequence: 
$$
0 \to U_{\text{red}} \to U \to \O_{\F}(C)\otimes \omega^{3}_{\pi} \to 0.
$$
Indeed, a section of $\O_{\F}(C)$ splits off a ruling line $L$ if and only it contains a subscheme of type $4p$ in $L$. 

The divisor $\delta_{\text{red}}$ is none other than the projective subbundle $\P U_{\text{red}}$.

\subsubsection{Jet bundle construction of $\delta_{\text{ram}}$}\label{deltaram}
A slightly more complicated jet bundle will be needed in the construction of $\delta_{\text{ram}}$. 

For this, we start by considering an affine set ${\bf A}^{2} \subset \F$ with coordinate ring $k[x,y]$.  We assume that $\pi$ is given by the inclusion of rings $k[x] \hookrightarrow k[x,y]$. We focus our attention at the point $(0,0)$. 

 Consider the ideal $J = (x, y^{3}) \cap (x^{2}, xy, y^{2}) = (x^{2}, xy, y^{3})$.  A function $f \in k[x,y]$ lies in $J$ if and only if $f$ is singular at $(0,0)$ and the scheme cut out by $f=0$ intersects the line $x=0$ with multiplicity at least $3$ at $(0,0)$.  These are precisely the sections of $\O_{\F}(C)$ which $\delta_{\text{ram}}$ parametrizes. Consider the following sequence: 
\begin{equation}\label{idealseq}
0 \to (x,y^{3})/J \to k[x,y]/J \to k[x,y]/(x,y^{3}) \to 0.
\end{equation}
The kernel is generated by the class $\bar{x}$, which we may view as a generator of the conormal bundle of the line $x=0$, restricted to the point $(0,0)$. 

We can globalize this as follows.  In $\F \times \F$, we take the chain of subschemes $\Delta \subset \F \times_{\P^1} \F \subset \F \times \F$, and consider the ideal sheaves $\cI_{\Delta}$ and $\cI_{\F \times_{\P^1} \F}$.  Then we consider the ideal  $$\cJ := (\cI_{\Delta}^{3} + \cI_{\F \times_{\P^1} \F}) \cap \cI_{\Delta}^{2} .$$

We define the relevant jet bundle as: $$J_{\text{ram}}^{4}(\O_{\F}(C)) :={\text{pr}}_{1*}( {\text{pr}}_{2}^{*}(\O_{\F}(C)) \otimes \O_{\F \times \F}/\cJ).$$ (The superscript $4$ indicates the rank of the bundle.) The sequence \eqref{idealseq}, globalized over $\F$ tells us that there is a sequence: 
\begin{equation}\label{J4}
0 \to \pi^{*}(\omega_{\P^{1}}) \otimes \O_{\F}(C) \to J_{\text{ram}}^{4}(\O_{\F}(C)) \to J_{\pi}^{3}(\O_{\F}(C)) \to 0.
\end{equation}

The sequence above implies the existence of a vector bundle $U_{\text{ram}}$ related to $U$ by a sequence: 
$$
0 \to U_{\text{ram}} \to U \to \pi^{*}(\omega_{\P^{1}}) \otimes \O_{\F}(C) \to 0.
$$

The divisor $\delta_{\text{ram}}$ is the projective subbundle $\P U_{\text{ram}}$.

\begin{proof}[Proof of \autoref{indep}]
As before, let $h$ denote the hyperplane class in $\P U$ which is pulled back from $\P^{N}$ via the projection $p_{1}$. From sequences \eqref{reljetseq} and \eqref{J4}, we find the following divisor class equalities in $A^{1}(\P U)$: 
\begin{eqnarray}
\delta_{\text{red}} &=& h + p_{2}^{*}(C + 3\omega_{\pi})\\
\delta_{\text{ram}} &=& h + p_{2}^{*}(C + \pi^{*}(\omega_{\P^{1}}) )
\end{eqnarray}
Of course $\delta$ is a multiple of $h$, so if $\delta, \delta_{\text{red}}$, and $\delta_{\text{ram}}$ were to be dependent, it would follow that $C+3\omega_{\pi}$ and $C + \pi^{*}(\omega_{\P^1})$ were proportional.  This is easily seen to be impossible for the curve classes $C$ we are considering.
\end{proof}

\begin{proof}[Proof of \autoref{divT}:]
Consider the diagram:

 \begin{equation} \label{sq}
\xymatrix @C -.6pc  {
\P U_{\text{sm}} \ar[rr]^{p_{1}} \ar[d]_{q_{\tT_{1}}}&&
 \P^{N}_{\text{sm}}  \ar[d]_{q}  \\
 \tT_{1}^{\circ} \ar[rr]^{\phi} && (\H^{\dagger})^{\circ}}
 \end{equation} 
The vertical arrows are quotients by $G = \Aut \F$, and the map $p_{1}$ is $G$-equivariant.  The kernel of the pullback map $q_{t} \from A^{*} (\tT_{1}^{\circ}) \to A^{*}(\P U_{\text{sm}})$ is generated by $c_{1}(N)$ where $N$ is some line bundle on $\tT_{1}^{\circ}$. However,  because $p_{1}$ is $G$-equivariant, this line bundle $N$ is the pullback of the analogous line bundle $M$ associated with the quotient $q$, hence its first Chern class is a multiple of $\kappa_{1}$ by \autoref{divH}.  Therefore,  we conclude \autoref{divT} from \autoref{indep}, which tells us that $A^{1}(\P U_{\text{sm}}) = 0$.

\end{proof}

\begin{theorem}\label{c2}
The Chern class $c_{2}(V)$ is a multiple of $\kappa_{1}^{2}$. 
\end{theorem}

\begin{proof}
  Let $\varphi_{*} \from A^{1}(\tT_{1}) \to A^{2}(\H)$ denote the
  proper push forward map.  \autoref{suppV} states that $c_{2}(V)$ is
  supported in $\cT_{1}$, and is therefore $\varphi_{*}(\alpha)$ for
  some cycle $\alpha \in A^{1}(\tT_{1})$. \autoref{divT} in turn
  states that $A^{1}(\tT_{1})$ is generated by
  $\varphi^{*}\kappa_{1}$. Thus, by the intersection-theoretic
  push-pull formula, we conclude that $c_{2}(V)$ is a multiple of
  $\kappa_{1} \cap \varphi_{*}[\tT_{1}] = \kappa_{1} \cap [\cT_{1}]$,
  which is in turn a multiple of $\kappa_{1}^{2}$, by
  \autoref{divH}.
\end{proof}

\section{The Chow ring of $\Hf$}\label{section:framed-chow-ring}

 In this  section we  prove the following proposition: 
\begin{proposition} \label{chow:dagger}
The Chow ring $A^{*}(\Hf)$ is generated by $\kappa_{1}$.
\end{proposition}

Our proof of \autoref{chow:dagger} requires the use of the  Maroni stratification of $\H^{\dagger}$. Let $\cN^{\dagger}_{k} \subset \Hf$ denote the locus of curves with Maroni invariant greater than or equal to $k$.  The main task in proving \autoref{chow:dagger} will be to show that the fundamental class $[\cN^{\dagger}_{k}]$ can be expressed in terms of kappa classes.

\subsection{The Maroni stratification} Every degree three cover $\alpha \from C \to \P^1$, has its {\sl reduced direct image bundle} $E_{\alpha}$ (see \cite{Miranda})  defined as: $$E_{\alpha} = \alpha_{*}\O_{C}/\O_{\P^1} .$$
The bundle $E_{\alpha}$ is a rank $2$, degree $-(g+2)$ vector bundle on the target $\P^1$, so it splits as a direct sum $\O(a) \oplus \O(b)$ where $a + b = -(g+2)$.  The Maroni invariant defined in \cite{maroni} is the integer $m = |b-a|$.

For each $k$, let $$\cN^{\dagger}_{k} \subset \H^{\dagger}$$ denote
the locally closed locus of covers having Maroni invariant at least
$k$.  The Chow groups
$A^{*}(\cN^{\dagger}_{k}\setminus \cN^{\dagger}_{k+2})$ are
tautological, by \cite[Thm.~3.3]{penev-vakil}. It remains to show that
the fundamental classes of the Maroni strata are tautological.

\subsection{The fundamental classes $[\cN^{\dagger}_{k}]$} 

\begin{proposition}\label{fun:mar} The fundamental classes $[\cN^{\dagger}_{k}] \in A^{*}(\H^{\dagger})$ are multiples of powers of $\kappa_{1}$.
\end{proposition}

Consider the universal diagram:
$$
\xymatrix @C -1.1pc  {
\cC \ar[rr]^{\alpha} \ar[rd]_f&&
\H^{\dagger} \times \P^{1} \ar[ld]^\pi \\
& \H^{\dagger}}
$$
  We let $E$ denote the universal reduced direct image bundle on $\H^{\dagger} \times \P^{1}$, and let $\sigma$ denote a horizontal section of $\pi$. 
  
 \begin{lemma} \label{push} All classes of the form $\pi_{*}(c_{1}(E)^{i} \cdot c_{2}(E)^{j} \cdot \sigma^{m})$  are multiples of powers of $\kappa_{1}$.
	\end{lemma}
	
        \begin{proof}
          In what follows, if $\epsilon$ is an element in a Chow ring, we will denote by $[\epsilon]_{d}$  its degree $d$ graded part.

          Write 
	\begin{eqnarray*}
	c_1(E) &=& -(g+2)\sigma + \pi^{*}(X)\\
	c_2(E) &=& \sigma \cdot \pi^{*}(Y) + \pi^{*}(Z)
	\end{eqnarray*}
	for some $X, Y \in A^{1}(\H^{\dagger})$ and $Z \in A^{2}(\H^{\dagger})$.  Now $X$ and $Y$ are known to be tautological from \autoref{divH}, so \autoref{push} will follow once we know that $Z = \pi_{*}(c_2(E)\cdot \sigma)$ is a multiple of $\kappa_{1}^{2}$.  (Note that the classes $\pi_{*}(c_1(E)^{i})$ are tautological for all $i$. In fact, they are multiples of $\kappa_{1}^{i-1}$.)
	
	To understand $Z$, we apply the Grothendieck-Riemann-Roch formula to $E$. Using the exact sequence 
	\begin{equation}\label{seqE}
0 \to \O_{\H^{\dagger} \times \P^{1}} \to \alpha_{*}\O_{\cC} \to E \to 0
	\end{equation}
we conclude that	
$$
	ch(\pi_{!}E) = -ch(R^{1}f_{*}\O_{\cC}) = \pi_{*}(ch(E)(1+\sigma)) = \pi_{*}(ch(E)) + \pi_{*}(ch(E) \cdot \sigma).
$$
	The first equality comes from sequence \eqref{seqE}, while the second is the content of Grothendieck-Riemann-Roch. 
	We take the degree 2 part of the second and last parts of this chain of equalities to get
	\begin{equation}\label{grr}
	-ch_{2}(R^{1}f_{*}\O_{\cC}) = \pi_{*}\left(\dfrac{c_1(E)^{3} - 3c_{1}(E)c_{2}(E)}{6}\right) + \pi_{*}\left( \left( \dfrac{c_1(E)^{2}}{2} - c_2(E) \right) \cdot \sigma \right)
	\end{equation}
The left hand side of \eqref{grr} is $0$. Indeed, by applying Grothendieck-Riemann-Roch to the structure sheaf $\O_{\cC}$ under the map $f$, we find that 
$$-[ch_{2}(R^{1}f_{*}\O_{\cC})]_{2} = f_{*}([td(T_{f})]_{3}),$$ but the degree three part of the Todd class of a line bundle is always trivial.

  When the right hand side of \eqref{grr} is written out, the coefficient of $Z$ is $\frac{g}{2}$, which is nonzero. The remaining terms on the right hand side of \eqref{grr} are multiples of $\kappa_{1}^{2}$, so the lemma follows.	
	\end{proof}
	
	Now we claim that the classes of the Maroni loci are linear combinations of classes of the form mentioned in \autoref{push}.

To see this, let's first consider the general situation where we have a rank $2$  vector bundle $F$ on $X \times \P^1$ where $X$ is a smooth variety, such that the generic splitting type of $F$ on the $\P^1$'s is  $\O \oplus \O$ (resp.\ $\O \oplus \O(1)$). Let $p \from X \times \P^1 \to X$ denote the projection, and fix a section $\sigma$ of $p$.

Now consider the restriction map on $X$:
$$
	\rho_{k} : p_{*}F(k\sigma) \to p_{*}(F(k\sigma)|_{\sigma})
$$
	This is a map between vector bundles on the open set $U_{k} \subset X$ where the splitting type of $F$ on the fiber of $p$ is at worst $\O(-(k+1)) \oplus \O((k+1))$ (resp.\ $\O(-(k+1)) \oplus \O((k+2))$). The closed subset $Y_{k} \subset U_{k}$ where $\rho_{k}$ drops rank is precisely the locus where $F$ is equal to $\O(-(k+1)) \oplus \O((k+1))$ (resp.\ $\O(-(k+1)) \oplus \O((k+2))$), and the fundamental class $[Y_{k}] \in A^{*}(U_{k})$ is an expression in the Chern classes of  $p_{*}F(k\sigma)$ and $p_{*}(F(k\sigma)|_{\sigma})$ by the Thom-Porteous formula \cite[Thm.~14.3]{fulton-intersectiontheory}.  So we arrive at the following lemma: 

	\begin{lemma}\label{port}
	The fundamental class of the locus $Y_{k} \subset U_{k}$ in $A^{*}(U_{k})$ over which $F$ splits as \\ $\O(-(k+1)) \oplus \O((k+1))$ or $\O(-(k+1)) \oplus \O((k+2))$  is expressible in terms of the Chern classes of $p_{*}F(k\sigma)$ and $p_{*}(F(k\sigma)|_{\sigma})$, provided that $Y_{k}$ has the expected codimension.  Furthermore, the Chern classes of $p_{*}F(k\sigma)$ and $p_{*}(F(k\sigma)|_{\sigma})$ are linear combinations of classes of the form $p_{*}((c_{1}F)^{k} \cdot (c_{2}F)^{j}\cdot [\sigma]^{m})$.
	\end{lemma}
	
	\begin{proof}[Proof of \autoref{port}:]
	Only the last sentence needs explanation, and it follows from Grothendieck-Riemann-Roch.
	\end{proof}
	
\begin{proof}[Proof of \autoref{fun:mar}:]
We apply \autoref{port} in our context with the following substitutions: The variety $X$ is  replaced by $\H^{\dagger}$ and the bundle $F$ is replaced by $E(\lceil \frac{g+2}{2} \rceil \sigma)$.  This twist of $E$ is chosen to ensure that the resulting bundle satisfies the assumed condition that the generic splitting type of $F$ is $\O \oplus \O$ or $\O \oplus \O(1)$.  \autoref{fun:mar} then follows from an application of \autoref{push} and the fact that the Maroni strata closures $\cN^{\dagger}_{n}$ are all of expected codimension (\autoref{mar:properties}).
	
\end{proof}

\begin{proof}[Proof of \autoref{chow:dagger}:] By \cite[Thm.~3.3]{penev-vakil}, the Chow groups $A^{*}(\cN^{\dagger}_{n} \setminus \cN^{\dagger}_{n+2})$ are tautological.  By iteratively applying the excision sequence for Chow groups, and using \autoref{fun:mar}, we conclude the theorem. 

\end{proof}

\section{Further directions}\label{further}

\subsection{The low degree regime}
The Maroni stratification allowed us to establish generators of the Chow ring, while relations were provided by emptiness of certain Hurwitz strata. 

It was crucial for us that the Maroni strata occurred in ``expected codimension''.  If we try to pursue the same strategy for the next two cases $\cH_{4,g}$ and $\cH_{5,g}$, we run into obstacles. The analogues of the Maroni stratification are not well-behaved -- the strata do not necessarily occur in expected codimension, at least in $\cH_{4,g}$ \cite[Example $4.3$]{dp:pic_345}.  (The notorious bielliptic locus in $\cH_{4,g}$ turns out to be such an ill-behaved stratum, but there are many more.)

\subsection{The high degree regime} In the range $d \geq 2g+1$, we see that $\Hdg$ is an open subset of a Grassmanian bundle $\G$ over $\Pic^{d}$, the universal degree $d$ Picard scheme over $\M$. 
The bundle $\G$ contains three divisors $\cT, {\mathcal D}, {\mathcal B}$ which parametrize pencils having triple ramification, $(2,2)$-ramification, and basepoints, respectively.  By definition, $\Hdg = {\G} \setminus {\mathcal B}$, and $\Hdg^{s} = {\G} \setminus \cT \cup {\mathcal D} \cup {\mathcal B}$.  

If we fix a codimension $k$, it appears as though the algebraic cohomology  $H_{\text{alg}}^{2k}(\Hdg^{s}, \Q)$ stabilizes to $0$ as $d$ and $g$ tend to $\infty$, with the slope $d/g$ greater than $2$. This has been verified for small $k$ by the first author, and uses the main result of \cite{madsen-weiss} which shows that $H^{*}(\M, \Q)$ stabilizes to the infinite polynomial ring $\Q[\kappa_{1},\kappa_{2}, ...]$.  The general case is the subject of ongoing work. Such a result should be viewed as an instance of similar stabilization phenomenon occuring for other moduli spaces: See \cite{tommasi-hypersurfaces, vakil-wood, church-stability}.  

In the same vein, if one considers the loci $\cT$ and ${\mathcal D}$ as ``discriminants'' in the space of maps $\Hdg$, it is reasonable to wonder whether they stabilize in the Grothendieck ring as $d, g \to \infty$.

\bibliographystyle{amsalpha}

\end{document}